\documentclass[11pt, leqno]{article}
\usepackage[]{geometry}
\usepackage[utf8]{inputenc}
\usepackage{tikz-cd}
\usepackage{cite}

\usepackage{amsmath,amsthm,amscd,amssymb,amsbsy,amstext,mathtools}
\usepackage{pdfpages}
\usepackage{enumerate}
\usepackage{mathrsfs}
\usepackage[utf8]{inputenc}
\usepackage{geometry}[margin=1in]
\numberwithin{equation}{section}
\usepackage{amsfonts}
\usepackage[T1]{fontenc}

\usepackage{hyperref}
\usepackage{url}

\usepackage{authblk}

\newtheorem{theorem}{Theorem}[section]
\newtheorem{lemma}[theorem]{Lemma}
\newtheorem{proposition}[theorem]{Proposition}

\newtheorem{problem}[theorem]{Problem}

\theoremstyle{definition}

\newtheorem{example}[theorem]{Example}

\theoremstyle{remark}
\newtheorem{remark}[theorem]{Remark}

\newcommand{\norm}[1]{\| {#1}\| }
\newcommand{\diam}{\mathsf{diam}}
\newcommand{\set}[1]{\left\{#1\right\}}

\newcommand{\abs}[1]{|#1|}

\newcommand{\Rn}{\mathbb{R}^n}

\newcommand{\grad}{\nabla}
\newcommand{\R}{\mathbb{R}}
\newcommand{\brac}[1]{\left(#1\right)}
\newcommand{\jbrac}[1]{\langle#1\rangle}
\newcommand{\ip}[1]{\jbrac{#1}}

\renewcommand{\epsilon}{\varepsilon}
\newcommand{\eps}{\epsilon}
\renewcommand{\subset}{\subseteq}

\renewcommand{\phi}{\varphi}
\newcommand{\T}{\mathbb T}
\newcommand{\qaq}{\text{\quad and \quad }}

\newcommand{\nsi}{N_\phi}

\newcommand{\LAcustom}[2]{%
  \begingroup
    \renewcommand{\theequation}{#2}%
    \refstepcounter{equation}%
    \text{(\theequation)}%
    \label{#1}%
  \endgroup
}
\newcommand{\LAQcustom}[3]{%
  \begin{itemize}
    \item[\LAcustom{#1}{#2}] #3
  \end{itemize}
}

\title{On a family of one-dimensional oscillation inequalities}
\author[1]{Fushuai Jiang}
\date{{\small
MSC 2020: 34C10, 46E35, 46E30
\\
\small Keywords: Sturm oscillation theorem, Gagliardo-Nirenberg estimate, uncertainty principle
}}
\affil[1]{Department of Mathematics, City University of Hong Kong}

\begin{document}
\maketitle

\begin{abstract}
Let $\varphi$ be a nonzero continuous mean-zero function on the one-dimensional torus and let $N_\varphi$ be the number of times that $\varphi$ changes signs. 
We prove the sharp family of oscillation inequalities of the types
\begin{equation*}
     N_\varphi\|\varphi\|_{\dot W^{-1,s}}
 \gtrsim_{p,s}
 \frac{\|\varphi\|_1^{1+p'/s}}{\|\varphi\|_p^{p'/s}}
 \, \, \text{ and } \, \,
 (N_\varphi)^\alpha\|\phi\|_{\dot W^{-1,s}}
 \gtrsim_{p,q,r,s,\alpha}
 \frac{\|\varphi\|_p\|\varphi\|_q}{\|\varphi\|_r}.
\end{equation*}
This resolves an open problem posed by S. Steinerberger and strengthens the original estimate. The proof is independent of optimal transport and is based on a Gagliardo-Nirenberg-type estimate as well as a quotient-space characterization of the negative Sobolev seminorm. As applications, we derive several oscillation estimates related to Fourier projection, the uncertainty principle, and the Sturm-Hurwitz theorem. 

\end{abstract}

\section{Introduction}

Let $\T$ denote the normalized one-dimensional torus. Throughout this paper, $\phi\in C(\T) = C(\T,\R)$ denotes a nonzero function with zero mean:
\begin{equation*}
    \int_\T \phi(x)d x=0.
\end{equation*}
In \cite[Theorem 4]{steinerberger2021wasserstein}, S.~Steinerberger proved the following Sturm oscillation-type uncertainty principle.
\begin{equation*} 
\# Z_\phi \left( \sum_{k=1}^{\infty} \frac{|\widehat{\phi}(k)|^2}{k^2} \right)^{1/2} \gtrsim \frac{\|\phi\|_{L^1(\mathbb T)}^2} {\|\phi\|_{L^\infty(\mathbb T)}}
\text{\quad where \quad}
Z_\phi:=\set{x\in\T:\phi(x)=0}.  
\end{equation*}
The Fourier sum on the left-hand side is, up to a normalization constant, equal to the homogeneous negative Sobolev norm
\begin{equation*} 
\left( \sum_{k=1}^{\infty} \frac{|\widehat{\phi}(k)|^2}{k^2} \right)^{1/2} \approx \|\phi\|_{\dot H^{-1}(\mathbb T)} = \|\phi\|_{\dot W^{-1,2}(\mathbb T)}. 
\end{equation*} 
Consequently, \eqref{eq:Steinerberger} may be written as 
\begin{equation} 
\# Z_\phi  \norm{\phi}_{\dot H^{-1}(\T)} \gtrsim \frac{
\norm{\phi}_{L^1(\mathbb T)}^2}{
\norm{\phi}_{L^\infty(\mathbb T)}}. 
\label{eq:Steinerberger}
\end{equation} 
On the other hand, the norm $\dot H^{-1}$ is also the infinitesimal metric associated with the quadratic Wasserstein distance. 
More precisely, if $\lambda$ denotes normalized Lebesgue measure on $\T$ and $t$ is so small that $1+t\phi\geq 0$, then 
\begin{equation} 
\mathcal W_2\brac{\lambda,(1+t\phi)\lambda} 
= 
\abs{t}\norm{\phi}_{\dot H^{-1}(\T)} + 
o(\abs{t}) 
\text{\quad as $t\to0$.}
\end{equation} 
As such, \eqref{eq:Steinerberger} can also be interpreted as an uncertainty principle for the linearized $\mathcal W_2$ geometry around the uniform measure; see \cite{peyre1,peyre2} for comparisons between $\mathcal W_2$ and weighted homogeneous negative Sobolev norms.

In this paper, we answer an open question posed in the same paper \cite[Page 317]{steinerberger2021wasserstein} by providing a sharp exponent classification of estimates of the type \eqref{eq:Steinerberger}. The main theorem is as below. 

\begin{theorem}\label{thm:main}
Fix $1\leq s<\infty$, $\alpha\geq1$, and $1\leq p,q,r\leq\infty$. Then
\begin{equation}
    \nsi^\alpha\norm{\phi}_{\dot W^{-1,s}(\T)}
 \gtrsim_{p,q,r,s,\alpha}
 \frac{\norm{\phi}_{L^p(\T)}\norm{\phi}_{L^q(\T)}}{\norm{\phi}_{L^r(\T)}}
 \label{eq:pqr}
\end{equation}
holds for all $\phi \in C(\T)$ with zero mean if and only if
\begin{equation}
 \alpha\geq1
 \qaq
 \frac{1}{p}+\frac{1}{q}-\frac{1}{r}\geq1+\frac{1}{s}.
 \label{eq:two num param}
\end{equation}
\end{theorem}

In Theorem \ref{thm:main} and the rest of the paper, $\nsi$ denotes the number of times that $\phi$ changes sign. We obviously have $\nsi \leq \#Z_\phi$, though the latter could be significantly larger. This is more or less inconsequential in classical Sturm-Liouville theory (see, e.g., \cite{zettl2005sturm}) as all interior zeroes of a nontrivial eigenfunction are simple. However, for general continuous functions, we make such a distinction, thinking of a sign change as a more stable event than having a zero. The classification for the exponent $\alpha$ adds little to the picture, but it helps control repeated zeros. 

The heart of the matter in proving Theorem \ref{thm:main} is to obtain the following critical-exponent $L^1$-$L^r$ estimate. 
\begin{equation}
     \nsi\norm{\phi}_{\dot W^{-1,s}(\T)}
 \gtrsim_{s,r}
 \frac{\norm{\phi}_{L^1(\T)}^{\beta}}
 {\norm{\phi}_{L^r(\T)}^{\beta-1}}
 \qaq
 \beta = 1 + \frac{r'}{s}.
 \label{eq:beta critical intro}
\end{equation}
Heuristically speaking, any mean-zero function with small negative Sobolev norm must change sign many times, unless its mass is highly concentrated. See Theorem \ref{thm:critical L1Lr}.
% As in \cite[Corollary to Theorem 4]{steinerberger2021wasserstein}, \eqref{eq:beta critical intro} also admits a formulation without reference to negative Sobolev spaces. Let $F:\T \to \R$ with mean $F_\T$, then
% \begin{equation}
%     (\text{\# monotonicity reversals of $F$})
%     \cdot \norm{F-F_\T}_{L^{s'}(\T)} \gtrsim \frac{\norm{F'}_{L^1(\T)}^{\beta}}{\norm{F'}_{L^{\beta - 1}(\T)}} \quad , \quad \beta = 1 + \frac{r'}{s}.
%     \label{eq:monotone statement}
% \end{equation}

The proofs leading to \eqref{eq:beta critical intro} and Theorem \ref{thm:main} will be given in Section \ref{sect:main}. Our argument bypasses the Wasserstein estimate in \cite{steinerberger2021wasserstein} and purely relies on a Gagliardo-Nirenberg-type estimate along with a convenient characterization of the negative-order Sobolev space $\dot W^{-1,s}$. Once we establish \eqref{eq:beta critical intro}, we can use a series of interpolation arguments to obtain Theorem \ref{thm:main}.

We also discuss a few applications of Theorem \ref{thm:main} and estimate \eqref{eq:beta critical intro}. They include 
\begin{itemize}
    \item low-frequency projection lower bounds in the Fourier-Lebesgue space (Theorem \ref{thm:low frequency}),
    \item an uncertainty principle (Theorem \ref{thm:uncertainty}), and
    \item an improved Sturm-Hurwitz oscillation bound with small low-frequency contamination (Theorem \ref{thm:sturm contamination}).
\end{itemize}

As a curious remark, we note that \eqref{eq:beta critical intro}, with $(r,s) = (\infty,2)$, is a strictly stronger estimate than \eqref{eq:Steinerberger} in principle, since $\norm{\phi}_{L^1(\Omega,\mu)}\leq \mu(\Omega) \norm{\phi}_{L^\infty(\Omega,\mu)}$ for a finite measure space $(\Omega,\mu)$. The situation in practice is a bit trickier. For the family of functions considered in \cite{steinerberger2021wasserstein}
\begin{equation*}
    \psi_{\eps,k}(x) = \eps \sin(k x) + \sin(k^2 x),
\end{equation*}
both estimates \eqref{eq:beta critical intro} and \eqref{eq:Steinerberger} predict $\approx k^2$ zeros, while the Sturm-Hurwitz theorem (see, e.g., \cite{Sturm1,Sturm2}) predicts $\gtrsim k$ zeros. On the other hand, we demonstrate in Section \ref{sect:zero count} of two sequences with the same exact number $2k$ of sign-changing zeros but very different mass distributions. The first $\set{\phi_{k}}$ is of the form
\begin{equation*}
    \phi_k(x) = -2k^2\pi\sin(2k\pi x)\exp\left[k^4(\cos(2 k \pi  x) - 1)\right],
\end{equation*}
For this family, Sturm-Hurwitz gives the exact lower bound, estimate \eqref{eq:beta critical intro} predicts $\gtrsim k$ zeroes, while \eqref{eq:Steinerberger} only gives $\gtrsim 1$. We contrast this with the family
\begin{equation*}
    \psi_k(x) = -2k^2\pi\sin(2k\pi x)\exp\left[k^4(\cos(2 \pi  x) - 1)\right].
\end{equation*}
The only change is $\cos(2k\pi x) \to \cos(2\pi x)$ inside the exponential.
Though also having exactly $2k$ sign changes, all significant mass is now confined to a single interval of length $\approx k^{-2}$, and the remaining zeros occur where the function is exponentially small. Consequently, our estimate \eqref{eq:beta critical intro} and Sturm-Hurwitz guarantee only $\gtrsim 1$ number of sign changes, while the earlier estimate \eqref{eq:Steinerberger} is asymptotically vacuous. 

These examples shed further light on both the nature and limitation of estimates of the type \eqref{eq:Steinerberger} and \eqref{eq:pqr}, as the negative-Sobolev inequality measures oscillations that carry substantial mass and is insensitive to small-amplitude oscillations.

Finally, we conjecture the higher-dimensional analogue of Theorem \ref{thm:main} in Section \ref{sect:higher dimensional}, for which the special case $(r,s) = (\infty,1)$ has been studied in \cite{steinerberger-interface,steinerberger2020metric, carroll2020enhanced}.

\subsection{Acknowledgment}

The author is supported by a grant from
the City University of Hong Kong (Project No. 7200844).

\subsection{Notation and convention}

Let $s,t, \ldots$ be a list of parameters. Let $A,B \geq 0$. We write $A\lesssim_{s,t, \ldots} B$ if there exists a constant $C \geq 0$ depending only on $s,t, \ldots$ such that $A\leq CB$. If $A\lesssim_{s,t, \ldots} B$ and $B\lesssim_{s,t, \ldots} A$, we write $A\approx_{s,t, \ldots} B$. When the list of parameters is clear from context, we omit the subscript $s,t, \ldots$.

Throughout this paper, we assume $\phi\in C(\T,\R)$ is nonzero with mean zero. We use $\nsi$ to denote the number of times $\phi$ changes sign. We shall assume $\nsi<\infty$, for otherwise, the inequalities of interest are vacuous.

% We use $\abs{\cdot}$ to denote the Lebesgue measure and $\#$ to denote the counting measure. 
For $1\leq s<\infty$, let $s' \in (1,\infty]$ be its H\"older conjugate. We use $W^{1,s'}(I)$ to denote the space of functions on an interval $I$ whose distributional derivative belongs to $L^{s'}(I)$. We define $W^{-1,s}$ via the usual duality pairing
\begin{equation*}
    \norm{\phi}_{\dot W^{-1,s}(\T)} := \sup_{u\in C^\infty(\T) ,\ 
    \norm{u'}_{L^{s'}(\T)}\leq 1
    }\abs{\ip{\phi,u}}.
\end{equation*}
However, we will use a much more convenient formula for $\dot W^{-1,s}$. See Proposition \ref{prop:sobolev norm} below.

\section{Main arguments}
\label{sect:main}

\subsection{A convenient Sobolev seminorm}

\begin{proposition}\label{prop:sobolev norm}
    Let $1 \leq s \leq \infty$ and let $s'$ be its H\"older conjugate. Let $\phi$ be a periodic distribution satisfying $\ip{\phi,1} = 0$ and assume that $\phi = F'$ for some $F \in L^s(\T)$. Then
    \begin{equation}
     \norm{\phi}_{\dot W^{-1,s}(\T)}
 =\inf_{c\in\R}\norm{F-c}_{L^s(\T)}.
 \label{eq:Sobolev norm}
\end{equation}
\end{proposition}

Proposition \ref{prop:sobolev norm} is a consequence of classical characterizations of $W^{-1,p}(\Omega)$. See for instance~\cite{SobolevNorm}. However, we give a short proof that exploits the quotient space structure, which we will encounter again in Theorem \ref{thm:GNS} and Remark \ref{rem:GNS - 1} below. See \cite[Lemma 6]{graham2020irregularity} for a related result.

\begin{proof}[Proof of Proposition \ref{prop:sobolev norm}]
    % Let $u$ be a periodic test function, so $\int_\T u' = 0$. Integration by part gives
    % \begin{equation*}
    %     \ip{\phi, u} = -\int_\T Fu' = -\int_\T (F-c)u' \leq \norm{F-c}_s\norm{u'}_{s'}.
    % \end{equation*}
    % Taking the supremum over all $\norm{u'}_{s'}\leq 1$ and taking the infimum over all $c \in \R$, we have
    % \begin{equation*}
    %     \norm{\phi}_{\dot W^{-1,s}} \leq \inf_{c\in \R}\norm{F-c}_s.
    % \end{equation*}

    For the moment, assume $1 \leq s < \infty$. Let $L^{s'}_\circ(\T)$ denote the space of mean-zero functions in $L^{s'}$. We have the identity
    \begin{equation*}
        \set{u' : u \in W^{1,s'}(\T)} = L^{s'}_\circ(\T) \cong (L^s/\R)^*
    \end{equation*}
    where the last isomorphism is also isometric. Therefore,
    \begin{equation*}
        \norm{\phi}_{\dot W^{-1,s}} = \sup_{g \in L^{s'}_\circ,\ \norm{g}_{s'}\leq 1} \int_\T Fg = \inf_{c\in \R}\norm{F-c}_s.
    \end{equation*}
    When $s = \infty$, we use the duality $(L^1_\circ(\T))^* \cong L^\infty(\T)/\R$, and the argument proceeds as before.
\end{proof}

\subsection{Gagliardo-Nirenberg-type estimate}

The heart of the matter is the following Gagliardo-Nirenberg-type estimate. Geometrically speaking, a transition of height $m$ with derivative $L^r$-size $A$ must occupy a horizontal scale at least $(m/A)^{r'}$. Its $L^s$ transport cost is therefore on the order of $m^s\left(\frac mA\right)^{r'}$.

\begin{lemma}\label{lem:GNS}
Let $I=[a,b]$, let $1\leq s<\infty$, let $1<r\leq\infty$, and let $r'$ be the H\"older conjugate of $r$. Let $u\in W^{1,r}(I) \cap C(I)$. Then
\begin{equation*}
    \inf_{c\in \R}\int_I|u-c|^sd x
 \gtrsim_{s,r} \frac{{\abs{u(b)-u(a)}}^{s+r'}}{\norm{u'}_{L^r(I)}^{r'}},
% c_{s,r}=4^{-(s+r')}
\end{equation*}
where we use the convention that $\frac{0}{0} = 0$ on the right-hand side.

\end{lemma}

% \begin{proof}
% Set $m = \abs{u(b)-u(c)}$.
% If $m = 0$, there is nothing to prove. Assume $m>0$ and fix an arbitrary $c \in \R$. Since
% \[
%  m\leq |u(a)-c|+|u(b)-c|,
% \]
% at least for one endpoint $x_0$, say $x_0 = a$,  we have
% \begin{equation}
%     \abs{u(a)-c}\geq \frac{m}{2}
% \end{equation}
% For $1<r<\infty$, Morrey-Sobolev embedding gives
% \[
%  |u(x)-u(a)|\leq \norm{u'}_{L^r(I)}|x-a|^{1/r'}.
% \]
% Also $m\leq \norm{u'}_{L^r(I)}|I|^{1/r'}$. Hence
% \[
% \xi:=\left(\frac{m}{4\norm{u'}_{L^r(I)}}\right)^{r'}\leq |I|.
% \]
% For $a\leq x\leq a+\xi$,
% \[
%  |u(x)-c|\geq\frac m2-\frac m4=\frac m4.
% \]
% Therefore
% \[
%  \int_I|u-c|^sd x
%  \geq \left(\frac m4\right)^s\ell
%  =4^{-(s+r')}\frac{m^{s+r'}}{\norm{u'}_{L^r(I)}^{r'}}.
% \]
% \end{proof}

\newcommand{\osc}{\mathsf{Osc}}
Lemma \ref{lem:GNS} is an immediate consequence of a more general higher-dimensional counterpart -- Theorem \ref{thm:GNS} -- below. See Remark \ref{rem:GNS - 1} after Theorem \ref{thm:GNS}. 

To fully state Theorem \ref{thm:GNS}, we lay out some standard terminology. Let $n \in \mathbb N_{> 0} $ and let $\R \ni r > n$. Morrey-Sobolev embedding gives $W^{1,r}(\Rn) \hookrightarrow C^{1-\frac{n}{r}}(\Rn)$, so every function in $W^{1,r}(\Rn)$ has a (H\"older) continuous representative. 

Recall that a domain $\Omega\subset \Rn$ is
\LAQcustom{extension domain}{$\Omega$-1}{a $W^{1,r}$-extension domain if for all $u \in W^{1,r}(\Omega)$, there exists $U \in W^{1,r}(\Rn)$ such that $U = u$ a.e. on $\Omega$ and $\norm{U}_{W^{1,r}(\Rn)}\lesssim_{\Omega,n,r} \norm{u}_{W^{1,r}(\Omega)}$;}
\LAQcustom{thick}{$\Omega$-2}{$n$-thick if $\abs{\Omega\cap B(x,\rho)} \gtrsim_{\Omega}\rho^n$ for all $x \in \overline{\Omega}$ and $0 \leq \rho \leq \diam \Omega$, where $\abs{\cdot}$ is the $n$-dimensional Lebesgue measure.}

For a continuous function $u \in C(\Omega)$, we define the oscillation of $u$ to be
\begin{equation*}
    \osc_{\Omega}(u) := \sup_{x,y \in \overline \Omega}\abs{u(x)-u(y)}.
\end{equation*}

\begin{theorem}\label{thm:GNS}
    Let $1\leq s < \infty$, $n < r \leq \infty$, and $\beta = \begin{cases}
        \frac{nr}{r-n} &\text{ if $r < \infty$}
        \\
        n &\text{ if $r = \infty$}
    \end{cases} $.
    Let $\Omega\subset \Rn$ be a domain satisfying \eqref{extension domain} and \eqref{thick}. Then for all $u \in W^{1,r}\cap C(\Omega)$, 
    \begin{equation}
        \inf_{c\in \R}\int_{\Omega} \abs{u(z) - c}^s dz \gtrsim_{\Omega,n,r,s} \frac{\osc_\Omega(u)^{s+\beta}}{\norm{\grad u}_{L^r(\Omega)}^\beta}.
        \label{eq:GNS}
    \end{equation}
    We use the convention $\frac{0}{0} = 0$ on the right-hand side of \eqref{eq:GNS}.
\end{theorem}

\begin{proof}[Proof of Theorem \ref{thm:GNS}]

For convenience, set 
\begin{equation*}
    \text{$A := \norm{\grad u}_{L^r(\Omega)}$ \quad and \quad $\alpha:= 1 - \frac{n}{r}$. }
\end{equation*}
 Thanks to \eqref{extension domain} the Morrey-Sobolev embedding, for all $z,w \in \overline{\Omega}$, 
\begin{equation}
    \abs{u(z)-u(w)} \leq C_{\Omega,n,r} A\cdot \abs{z-w}^{\alpha} \leq C_{\Omega,n,r} A\cdot (\diam \Omega)^\alpha.
    \label{eq:Sobolev embedding}
\end{equation}

Let $x,y \in \overline\Omega$ and $c\in \R$ be arbitrary, and set $m := \abs{u(x) - u(y)}$. We have
\begin{equation*}
    m = \abs{u(x) - u(y)} \leq \abs{u(x) - c} + \abs{u(y) - c}.
\end{equation*}
Without loss of generality, we may assume
\begin{equation}
    \abs{u(x) - c} \geq \frac{m}{2}.
    \label{eq:ux - c}
\end{equation}
Set
\begin{equation*}
    \rho := \brac{\frac{m}{2C_{\Omega,n,r}A}}^{1/\alpha}
    \label{eq:rho}
\end{equation*}
By \eqref{eq:Sobolev embedding}, we see that
\begin{equation*}
    0 \leq \rho \leq \diam \Omega.
\end{equation*}
Therefore, for all $z \in \Omega \cap B(x,\rho)$, we have
\begin{equation}
    \abs{u(z)-u(x)} \leq C_{\Omega,n,r}A\rho^\alpha = \frac{m}{4}.
    \label{eq:uz - ux}
\end{equation}
Then \eqref{eq:ux - c} and \eqref{eq:uz - ux} imply
\begin{equation*}
    \abs{u(z) - c} \geq \abs{u(x) - c} - \abs{u(z) - u(x)} \geq \frac{m}{2}-\frac{m}{4} = \frac{m}{4}
    \text{\quad for all $z \in \Omega \cap B(x,\rho)$.}
\end{equation*}
Thanks to assumption \eqref{thick}, 
\begin{equation}
    \begin{split}
        \int_{\Omega}\abs{u(z)-c}^sdz \geq \int_{\Omega \cap B(x,\rho)}\abs{u(z)-c}^sdz 
        &\geq \brac{\frac{m}{4}}^s\abs{\Omega \cap B(x,\rho)} 
        \\
        &\gtrsim_{\Omega} m^s\rho^n \approx_{\Omega,n,r} m^s\brac{\frac{m}{A}}^{n/\alpha}.
    \end{split}
    \label{eq:integral}
\end{equation}
Since $x,y \in \overline{\Omega}$ and $c \in \R$ are arbitrary, the estimate \eqref{eq:GNS} then follows from \eqref{eq:rho}, \eqref{eq:integral}, and the fact that $n/\alpha = \beta$.
\end{proof}

\begin{remark}\label{rem:GNS - 1}
    To see how Theorem \ref{thm:GNS} implies Lemma \ref{lem:GNS}, we note the following simple observations: 
    \begin{itemize}
        \item $\beta$ is the H\"older conjugate of $r$ when $n = 1$ and $\osc_I(u) \geq \abs{u(b) - u(a)}$;
        \item $I = [a,b]\subset \R$ is a $W^{1,r}$ extension domain for $r > 1$ (see, for instance,~\cite{stein-singular-integral}) and is $1$-thick.
    \end{itemize}
\end{remark}

\begin{remark}
    Recall Proposition \ref{prop:sobolev norm}.
    We note that \eqref{eq:GNS} has a natural quotient-space form that resembles the classical $L^\infty$ Gagliardo-Nirenberg estimate
    \begin{equation*}
        \norm{[u]}_{L^\infty(\Omega)/ \R} \lesssim_{\Omega,n,s,r} \norm{[u]}_{L^s(\Omega)/ \R}^{1-\theta}\norm{\grad u}_{L^r(\Omega)}^\theta, \quad \theta = \frac{\beta}{s+\beta}.
    \end{equation*}
    In the above, $[u] = \set{u+c: c \in \R}$, and we have natural choices of norms
    \begin{equation*}
        \text{$\norm{[u]}_{L^s/\R}:= \inf\limits_{c\in \R}\norm{u - c}_{L^s}$ \quad and \quad  $\norm{[u]}_{L^\infty/\R} := \frac{1}{2}\osc(u)$.}
    \end{equation*}
    For related estimates, see also~\cite{GNSrelated,GNSrelated2}. 
\end{remark}

\subsection{Proofs of main theorems}

\begin{lemma}\label{lem:sum}
Let $1\leq s<\infty$, $1<r<\infty$, and let $r'$ be the H\"older conjugate of $r$. For positive numbers $m_j,A_j$, $1\leq j\leq N$,
\begin{equation}
    \sum_{j=1}^N\frac{m_j^{s+r'}}{A_j^{r'}}
 \geq
 \frac{1}{N^s}
 \frac{\displaystyle\brac{\sum_{j=1}^Nm_j}^{s+r'}}
 {\displaystyle\brac{\sum_{j=1}^NA_j^r}^{r'/r}}.
 \label{eq:aggregate}
\end{equation}
For $r = \infty$, we have
\begin{equation*}
    \sum_{j=1}^N\frac{m_j^{s+r'}}{A_j^{r'}}
 \geq
 \frac{1}{N^s}
 \frac{\displaystyle\brac{\sum_{j=1}^Nm_j}^{s+r'}}
 {\max\limits_{j=1, \cdots, N}A_j}.
\end{equation*}
\end{lemma}

\begin{proof}
We prove the case $r < \infty$. The modification $r = \infty$ is standard. Set
\[
 \delta=1+s\left(1-\frac1r\right) = 1 + \frac{s}{r'}
 % \frac{(r-1)(s+r')}{r} 
 \in (1, 1+s).
\]
We have
\begin{equation*}
    m_j^\delta=
 \left(\frac{m_j^{s+r'}}{A_j^{r'}}\right)^{1/r'}
 (A_j^r)^{1/r}.
\end{equation*}
By H\"older's inequality,
\begin{equation}
    \sum_jm_j^\delta
 \leq
 \left(\sum_j\frac{m_j^{s+r'}}{A_j^{r'}}\right)^{1/r'}
 \left(\sum_jA_j^r\right)^{1/r}.
 \label{eq:holder}
\end{equation}
Using the convexity of the function $t \mapsto t^{\delta}$, we have
\begin{equation}
    \frac{1}{N}\sum_{j=1}^N m_j^\delta \geq \brac{\frac{1}{N}\sum_{j=1}^Nm_j}^\delta.
    \label{eq:jensen}
\end{equation}
Rearranging \eqref{eq:holder} and using \eqref{eq:jensen}, 
we have
\begin{equation*}
    \sum_j\frac{m_j^{s+r'}}{A_j^{r'}}
 \geq
 \frac{(\sum_jm_j^\delta)^{r'}}{(\sum_jA_j^r)^{r'/r}}
 \geq
 \frac{N^{(1-\delta)r'}(\sum_jm_j)^{\delta r'}}
 {(\sum_jA_j^r)^{r'/r}}.
\end{equation*}
Since $(\delta-1)r'=s$ and $\delta r'=s+r'$, we see that \eqref{eq:aggregate} holds.
\end{proof}

We start with the critical-exponent estimate \eqref{eq:beta critical intro} from the introduction. 
\begin{theorem}\label{thm:critical L1Lr}
Let $1\leq s<\infty$, $1<r\leq\infty$, and let $r'$ be the H\"older conjugate of $r$. Define
\begin{equation}
    \beta_c = \beta_c(s,r)=1+\frac{r'}s.
 \label{eq:beta critical}
\end{equation}
Then
\[
 \nsi\norm{\phi}_{\dot W^{-1,s}}
 \gtrsim_{s,r}
 \frac{\norm{\phi}_1^{\beta_c}}
 {\norm{\phi}_r^{\beta_c-1}}.
\]
\end{theorem}

\begin{proof}
Without loss of generality, we assume $\phi$ is continuous. Let $I_j=(a_j,b_j)$ be the sign intervals of $\phi$ such that $\phi$ is not identically zero. Set
\begin{equation*}
     m_j=\norm{\phi}_{L^1(I_j)}
 \qaq
 A_j=\norm{\phi}_{L^r(I_j)}.
\end{equation*}
For the moment, we assume $r < \infty$. We have
\begin{equation}
     \sum_jm_j=\norm\phi_1
     \qaq
 \sum_jA_j^r=\norm\phi_r^r.
 \label{eq:mj and Aj}
\end{equation}
Let $F$ be any periodic primitive of $\phi$. Because $F'=\phi$ has constant sign on $I_j$,
\[
 \abs{F(b_j)-F(a_j)}=m_j.
\]
For an arbitrary constant $c$, Lemma~\ref{lem:GNS} gives
\begin{equation}
     \norm{F-c}_s^s
 =\sum_j\int_{I_j}|F-c|^s
 \gtrsim_{s,r}
 \sum_j\frac{m_j^{s+r'}}{A_j^{r'}}.
 \label{eq:F-c}
\end{equation}
Using Lemma~\ref{lem:sum} and \eqref{eq:mj and Aj} to estimate \eqref{eq:F-c}, we have
\[
 \norm{F-c}_s^s
 \gtrsim_{s,r}
 \frac{\norm\phi_1^{s+r'}}{\nsi^s\norm\phi_r^{r'}}.
\]
Taking the infimum over all $c \in \R$ and using the convient formula for $\dot W^{-1,s}$ from Proposition \ref{prop:sobolev norm}, we see that Theorem \ref{thm:critical L1Lr} holds for $r < \infty$.

For $r = \infty$, we replace $\sum_j A_j^r$ by $\max_j A_j$, and the same $L^\infty$ part of Lemma~\ref{lem:sum} applies. 

This concludes the proof of the theorem. 
\end{proof}

\begin{remark}\label{rem:sharp critical cases}
    We consider some special cases.
    \begin{itemize}
        \item For $s=1$ and $r=\infty$, it gives
        \begin{equation*}
            \nsi\norm\phi_{\dot W^{-1,1}}
 \gtrsim\frac{\norm\phi_1^2}{\norm\phi_\infty}.
        \end{equation*}
        By the Kantorovich-Rubinstein duality, this gives the same transport inequality given obtained in~\cite[proof of Theorem 4]{steinerberger2021wasserstein}, assuming solutions to $\psi = \psi_\T := \int_\T \psi$ are nondegenerate:
        \begin{equation*}
            \#\set{\psi = \psi_\T} \cdot \mathcal W_1(\psi(x)dx, \psi_\T dx) \gtrsim \frac{\norm{\psi - \psi_\T}_1^2}{\norm{\psi}_\infty}.
        \end{equation*}
        \item For $s=2$ and $r=\infty$, Theorem~\ref{thm:critical L1Lr} gives
        \begin{equation*}
            \nsi\norm\phi_{\dot H^{-1}}
 \gtrsim\frac{\norm\phi_1^{3/2}}{\norm\phi_\infty^{1/2}} = \sqrt{r_\phi}\frac{\norm \phi_1^{2}}{\norm \phi_\infty}, \quad r_\phi = \frac{\norm \phi_\infty}{\norm \phi_1} \geq 1.
        \end{equation*}
        % As mentioned in the introduction, this is stronger than the estimate \eqref{eq:Steinerberger} obtained in~\cite{steinerberger2021wasserstein}.
\item As $s\to\infty$, the critical exponent formally tends to one, recovering the following heuristic estimate, analogous version of which is also considered in~\cite{steinerberger2021wasserstein}:
\begin{equation*}
    \inf_{c \in \R}\norm{F-c}_\infty\gtrsim\frac{\norm\phi_1}{\nsi}.
\end{equation*}
    \end{itemize}

\end{remark}

\begin{theorem}\label{thm:beta interpolate}
Fix $1\leq s<\infty$, $\alpha,\beta\geq1$, and $1\leq p,q, r\leq\infty$. The estimate
\begin{equation}
     (\nsi)^\alpha\norm{\phi}_{\dot W^{-1,s}}
 \gtrsim_{p,r,s,\alpha}
 \frac{\norm{\phi}_p^\beta}{\norm{\phi}_r^{\beta-1}}
 \label{eq:one num}
\end{equation}
holds uniformly if and only if
\begin{equation} 
\text{
$\alpha \geq 1$, $p < r$, and $\beta\left(\frac{1}{p}-\frac{1}{r}\right)
 \geq1+\frac{1}{s}-\frac{1}{r}$.
}
 \label{eq:one num param}
\end{equation}
\end{theorem}

\begin{proof}
We begin with sufficiency. Let
\[
 \theta=\frac{\frac{1}{p}-\frac{1}{r}}{1-\frac{1}{r}}.
\]
Riesz-Thorin interpolation yields
\begin{equation*}  \norm{\phi}_p\leq\norm{\phi}_1^\theta\norm{\phi}_r^{1-\theta}
 \quad \implies \quad 
 \frac{\norm\phi_p^\beta}{\norm\phi_r^{\beta-1}}
 \leq
 \frac{\norm{\phi}_1^{\beta\theta}}{\norm{\phi}_r^{\beta\theta-1}}.
\end{equation*}
The last part of condition \eqref{eq:one num param} is precisely $\beta\theta\geq\beta_c(s,r)$ from \eqref{eq:beta critical}. 
Since $\T$ is normalized, we have $\norm{\phi}_1 \leq \norm{\phi}_r$. Since $\alpha \geq 1$, we have $\nsi^\alpha \geq \nsi$.  
The estimate \eqref{eq:one num} then follows from Theorem~\ref{thm:critical L1Lr}.

For necessity, set $\phi_n(x)=\sin(2\pi nx)$, and observe that
\begin{equation}
    N_{\phi_n} \approx n,
 \qquad
 \norm{\phi_n}_{\dot W^{-1,s}} \approx n^{-1},
 \qquad
 \norm{\phi_n}_t\approx_t 1.
 \label{eq:test - high frequency}
\end{equation}
Substituting \eqref{eq:test - high frequency} into \eqref{eq:one num} shows that $\alpha\geq1$ is necessary.

Next, choose a fixed continuous mean-zero function $\psi$ compactly supported in $(0,1)$ with a bounded number of sign intervals and set
\begin{equation*}
    \phi_\eps(x)=\psi\left(\frac{x-1/2}{\eps}\right).
\end{equation*}
We then have
\begin{equation}
    N_{\phi_\eps}\approx1
    \quad,\quad
    \norm{\phi_\varepsilon}_p\approx\varepsilon^{1/p}
 \text{\quad, and \quad}
 \norm{\phi_\varepsilon}_{\dot W^{-1,s}}
 \approx\varepsilon^{1+1/s}.
 \label{eq:test - concentrate}
\end{equation}
Substitution \eqref{eq:test - concentrate} into \eqref{eq:one num} forces
\begin{equation*}
    \frac{\beta}{ r}-\frac{\beta-1}{r}\geq1+\frac{1}{s}.
\end{equation*}
Therefore, \eqref{eq:one num param} must hold.
\end{proof}

\begin{proof}[Proof of Theorem \ref{thm:main}]
For sufficiency, we note that $1\leq p,q \leq r$ by \eqref{eq:two num param}, so we may interpolate $L^p$ and $L^q$ separately between $L^1$ and $L^r$. Put
\[
 \theta_p=\frac{\frac{1}{p}-\frac{1}{r}}{1-\frac{1}{r}}
 \qaq
 \theta_q=\frac{\frac{1}{q}-\frac{1}{r}}{1-\frac{1}{r}}.
\]
Also set $\gamma=\theta_p+\theta_q$. Then
\[
 \frac{\norm{\phi}_p\norm{\phi}_q}{\norm\phi_r}
 \leq
 \frac{\norm\phi_1^\gamma}{\norm\phi_r^{\gamma-1}}.
\]
Condition \eqref{eq:two num param} is equivalent to $\gamma\geq\beta_c(s,r)$ from \eqref{eq:beta critical}, so Theorem~\ref{thm:critical L1Lr} applies. Necessity follows from the same two tests in the proof of Theorem \ref{thm:beta interpolate}, so we omit the details.
\end{proof}

\section{Application and further discussions}
\label{sect:additional}

\newcommand{\Rphi}{\mathcal R(\phi)}
Recall Theorem \ref{thm:main} and Theorem \ref{thm:critical L1Lr}.
For the rest of this section, we assume $2 \leq s < \infty$ and define
\begin{equation}
    \Rphi := \begin{cases}
        \frac{\norm{\phi}_p\norm{\phi}_q}{\norm{\phi}_r}
        &\text{ where } \frac{1}{p} + \frac{1}{q} - \frac{1}{r} \geq 1 + \frac{1}{s}
        \\[1em]
        \text{ or }
        \\[1em]
        \mathcal R(\phi) = \frac{\norm\phi_{1}^{1 + r'/s}}{\norm\phi_r^{r'/s}}
        &\text{ where } r \in (1,\infty] 
    \end{cases} \ .
    \label{eq:Rphi}
\end{equation}

\newcommand{\FL}{\mathcal FL}

\subsection{$L^p$ bounds of low frequency Fourier projection and uncertainty principle}

Recall that the $L^p$ Fourier-Lebesgue seminorm of order $m\in \R$ of a zero-mean function $\phi$ is given by
\begin{equation*}
    \norm{\phi}_{\FL^{p}_m(\T)}:= \brac{\sum_{k\in \mathbb Z\setminus 0}\abs{k}^{mp}
    \abs{\hat\phi(k)}^{p} 
    }^{1/p}.
\end{equation*}
In particular, $\norm{\phi}_{\FL^{p}_{-1}} = \brac{\sum_{k\neq 0}\frac{\abs{\hat\phi(k)}^{p}}{\abs{k}^{p}}}^{1/p}$. Let $\mathcal P_n:L^2(\T) \to L^2(\T)$ denote the projection onto the first $n$ Fourier modes, so 
\begin{equation*}
    \mathcal P_n\phi(x):= \sum_{0 < \abs{k}\leq n}\hat{\phi}(k)e^{2\pi i kx}
    \text{\quad for a mean zero $\phi \in L^2(\T)$.}
\end{equation*}

\begin{theorem}\label{thm:low frequency}
Let $2\leq s < \infty$.
There exists a constant $c_{p,q,r,s}$ such that for all $n \geq 1$, $m \geq 0$, and nonzero $\phi \in \FL^{s'}_m\cap C(\T)$ with zero mean, we have
\begin{equation}
    \norm{\mathcal P_n\phi}_{\FL^{s'}_{-1}(\T)} \geq c_{p,q,r,s}\frac{\mathcal R(\phi)}{\nsi} - n^{-(m+1)}\norm{\phi}_{\FL^{s'}_{m}(\T)}
    \label{eq:low frequency}
\end{equation}
with $\Rphi$ as in \eqref{eq:Rphi}.
\end{theorem}

\begin{proof}
    We prove \eqref{eq:low frequency} for $\Rphi = \frac{\norm{\phi}_p\norm{\phi}_q}{\norm{\phi}_r}$. The other bound is similar using Theorem \ref{thm:critical L1Lr}.

    Let $F$ be a mean-zero periodic primitive of $\phi$, so $\hat F(k) = \frac{\hat \phi(k)}{2\pi i k}$ for $k \neq 0$.
    Since $2 \leq s < \infty$, we have $s' \in (1,2]$. By the (reverse) Hausdorff-Young inequality and \eqref{eq:Sobolev norm}, 
    \begin{equation*}
        \norm{\phi}_{\dot W^{-1,s}} = \inf_{c\in \R}\norm{F-c}_s \leq \norm{F}_s \leq \brac{\sum_{k\neq 0}\abs{\hat F(k)}^{s'}}^{1/s'} = \frac{1}{2\pi}\brac{\sum_{k\neq 0}
        \frac{\abs{\hat \phi(k)}^{s'}}{\abs{k}^{s'}}
        }^{1/s'}.
    \end{equation*}
    Combining this with Theorem \ref{thm:main} yields
    \begin{equation}
        \norm{\phi}_{\FL^{s'}_{-1}} = \brac{\sum_{k\neq 0}
        \frac{\abs{\hat \phi(k)}^{s'}}{\abs{k}^{s'}}
        }^{1/s'} \gtrsim_{p,q,r,s}\frac{\mathcal A_{p,q,r}(\phi)}{\nsi}.
        \label{eq:low freq - 1}
    \end{equation}

    By the triangle inequality,
    \begin{equation}
        \begin{split}
            \norm{\phi}_{\FL^{s'}_{-1}} 
            &\leq \norm{\mathcal P_n\phi}_{\FL^{s'}_{-1}} + 
            \brac{\sum_{\abs{k} > n}
            \frac{\abs{\hat \phi(k)}^{s'}}{\abs{k}^{s'}}
            }^{1/s'}
            \\
            &
            \leq 
            \norm{\mathcal P_n\phi}_{\FL^{s'}_{-1}} +
            \brac{
            \sum_{\abs{k} > n}
            \frac{
            \abs{k}^{ms'}\abs{\hat\phi(k)}^{s'}
            }{
            \abs{k}^{(m+1)s'}
            }
            }^{1/s'}
            \\
            &\leq
            \norm{\mathcal P_n\phi}_{\FL^{s'}_{-1}} + 
            \brac{
            n^{-(m+1)s'}\norm{\phi}_{\FL^{s'}_m}^{s'}
            }^{1/s'} 
            = \norm{\mathcal P_n\phi}_{\FL^{s'}_{-1}} + n^{-(m+1)}\norm{\phi}_{\FL^{s'}_m}.
        \end{split} 
        \label{eq:low freq - 2}
    \end{equation}
    Estimate \eqref{eq:low frequency} then follows from \eqref{eq:low freq - 1} and \eqref{eq:low freq - 2}.
    
\end{proof}

% \begin{theorem}
%     Let $s \geq 2$. For all mean-zero $\phi \in C^\infty(\T)$, we have
%     \begin{equation}
%         \norm{\mathcal P_n\phi}_{L^s(\T)} \gtrsim_{s} \frac{1}{N_\phi} \frac{\norm{\phi}_{L^1(\T)}^{1+1/s}}{\norm{\phi}_{L^\infty(\T)}^{1/s}}
%         \text{\quad for sufficiently large $n$.}
%     \end{equation}
%     In particular, for $s = 2$, we have
%     \begin{equation*}
%         \norm{\mathcal P_n\phi}_{L^2(\T)} \gtrsim \frac{1}{N_\phi} \frac{\norm{\phi}_{L^1(\T)}^{3/2}}{\norm{\phi}_{L^\infty(\T)}^{1/2}}
%         \text{\quad for sufficiently large $n$.}
%     \end{equation*}
% \end{theorem}

A closely related result to Theorem \ref{thm:low frequency} is the following uncertainty principle that resembles \eqref{eq:Steinerberger}. Heuristically speaking, if $\phi$ has few sign changes, then it must have large $\dot W^{-1,s}$ discrepancy. 

\begin{theorem}\label{thm:uncertainty}
    Let $2\leq s < \infty$.
    For all $\phi \in C(\T)$ with zero mean, $M\geq \norm{\phi_-}_{L^\infty(\T)}$, and $n \in \mathbb N_{>0}$,
    \begin{equation}
        \nsi \left[
        \frac{M}{n} + \brac{\sum_{k=1}^{n-1}
        \frac{\abs{\hat\phi(k)}^{s'}}{k^{s'}}
        }^{1/s'}
        \right]
        \gtrsim_{p,q,r,s}
        \Rphi
        \label{eq:ET main}
    \end{equation}
    with $\Rphi$ as in \eqref{eq:Rphi}.
\end{theorem}

We will make use of the following $L^p$ Er\"os-Tur\'an inequality~\cite[Proposition 4]{graham2020irregularity}.

\begin{theorem}[\hspace{1sp}\cite{graham2020irregularity}]\label{thm:ET graham}
Let $1 < r \leq \infty$, $2\leq s < \infty$, and let $r',s'$ be their H\"older conjugates. Let $n\in\mathbb N$, let $\mu\in\mathcal P(\T)$ be a probability measure, and let $\lambda$ be normalized Lebesgue measure. Then
\begin{equation}
    \norm{\mu-\lambda}_{\dot W^{-1,s}(\T)}
 \lesssim_s \frac1n+
 \left(\sum_{k=1}^{n-1}
 \frac{|\widehat\mu(k)|^{s'}}{k^{s'}}\right)^{1/s'}.
    \label{eq:LpET}
\end{equation}
\end{theorem}

\begin{proof}[Proof of Theorem \ref{thm:uncertainty}]

We prove \eqref{eq:ET main} for $\Rphi = \frac{\norm{\phi}_p\norm{\phi}_q}{\norm{\phi}_r}$. The other bound is similar using Theorem \ref{thm:critical L1Lr}.

Let $M$ be as in the hypothesis. Define
\begin{equation}
 d\mu_M(x)=\left(1+\frac{\phi(x)}M\right)d x.
 \label{eq:dmuM}
\end{equation}
The density in \eqref{eq:dmuM} is nonnegative by the choice of $M$. Since $\int_\T\phi=0$,
\[
 \mu_M(\T)=\int_\T\left(1+\frac\phi M\right)d x=1,
\]
so $\mu_M$ is a probability measure. Homogeneity of the negative Sobolev norm gives
\begin{equation}
     \norm{\mu_M-\lambda}_{\dot W^{-1,s}}
 =\frac1M\norm\phi_{\dot W^{-1,s}}.
 \label{eq:normalization - hom}
\end{equation}
For $k\neq0$, we have $\widehat\lambda(k)=0$ , so
\begin{equation}
     \widehat{\mu_M}(k)=\frac{\widehat\phi(k)}M.
     \label{eq:normalization - hom - 2}
\end{equation}
Substitute \eqref{eq:normalization - hom} and \eqref{eq:normalization - hom - 2} into \eqref{eq:LpET} and multiply by $M$, we have
\begin{equation}
    \norm{\phi}_{\dot W^{-1,s}}
 \lesssim_s \frac{M}{n}+
 \left(\sum_{k=1}^{n-1}
 \frac{|\widehat\phi(k)|^{s'}}{k^{s'}}\right)^{1/s'}
 \text{\quad for $s\in [2,\infty)$.}
 \label{eq:normalization key}
\end{equation}

Combining Theorem \ref{thm:main} and Lemma \eqref{eq:normalization key}, we have
\[
 \frac{\Rphi}{\nsi}
 \lesssim_s
 \norm\phi_{\dot W^{-1,s}}
 \lesssim_s
 \frac{M}{n}+
 \left(\sum_{k=1}^{n-1}
 \frac{\abs{\hat\phi(k)}^{s'}}{k^{s'}}\right)^{1/s'}.
\]
This completes the proof.
\end{proof}

\begin{remark}
    We note a subtle difference between Theorem \ref{thm:low frequency} and Theorem \ref{thm:uncertainty}. In Theorem \ref{thm:low frequency}, we assume that $\phi \in \FL^{s'}_m$ and control the high-frequency tail by $n^{-(m+1)}\norm{\phi}_{\FL^{s'}_m}$, which improves for smoother $\phi$. On the other hand, Theorem \ref{thm:uncertainty} demands minimal regularity for $\phi$, so the tail (discrepancy) is a fixed error $\frac{\norm{\phi_-}_\infty}{n}$ which is independent of higher regularity of $\phi$. 
\end{remark}

\subsection{Sturm-Hurwitz theorem with small perturbation}

Recall that classical Sturm-Hurwitz Theorem~\cite{Sturm1,Sturm2}.

\begin{theorem}[Sturm-Hurwitz Theorem]
    If $u$ is a nonzero continuous real-valued function on $\T$ with a spectral gap $[-M,M]$, $M > 0$, then $u$ has at least $2M$ sign changes on $\T$. 
\end{theorem}

The Sturm-Hurwitz theorem involves an exact spectral gap. We prove a robust version of this theorem when low-frequency contamination is present but small in the Fourier-Lebesgue seminorm. 

\begin{theorem}\label{thm:sturm contamination}
    Let $2\leq s < \infty$.
    Suppose a zero-mean function $\phi \in C(\T)$ can be written as
    \begin{equation}
        \phi = h + g
    \text{\quad where $\hat h(k) = 0$ for $0 < \abs{k} < M$, and $\hat h, \hat g \in \ell^{s'}(\mathbb Z)$.}
    \label{eq:spectral gap}
    \end{equation}
    Then
    \begin{equation*}
        \nsi \gtrsim_{s,r} \frac{\mathcal R(\phi)}{
        M^{-1}\norm{\hat h}_{\ell^{s'}}
         + \norm{g}_{\FL^{s'}_{-1}}
        }
    \end{equation*}
    with $\Rphi$ as in \eqref{eq:Rphi}.
    
    In particular, if $\norm{\phi}_1 \approx \norm{\phi}_r \approx \norm{\hat h}_{\ell^{s'}} \approx 1$ and $\norm{g}_{\FL^{s'}_{-1}}\lesssim M^{-1}$, then $\nsi \gtrsim M$. 
\end{theorem}

\begin{proof}
    We prove the estimate for $\Rphi = \frac{\norm{\phi}_p\norm{\phi}_q}{\norm{\phi}_r}$. The other bound is similar using Theorem \ref{thm:critical L1Lr}.
    
    By Theorem \ref{thm:main}, we have
    \begin{equation}
        N_\phi \gtrsim_{s,r}\frac{\Rphi}{\norm{\phi}_{\dot W^{-1,s}}}.
        \label{eq: low freq contamination - 0}
    \end{equation}
    As in the proof of Theorem \ref{thm:low frequency}, we invoke the (reverse) Hausdorff-Young theorem to obtain
    \begin{equation}
        \norm{\phi}_{\dot W^{-1,s}} \leq \frac{1}{2\pi} \norm{\phi}_{\FL^{s'}_{-1}} \leq 
        \frac{1}{2\pi}\brac{\norm{h}_{\FL^{s'}_{-1}} + \norm{g}_{\FL^{s'}_{-1}}}.
        \label{eq: low freq contamination - 1}
    \end{equation}
    By the spectral gap assumption in \eqref{eq:spectral gap}, 
    \begin{equation}
        \norm{h}_{\FL^{s'}_{-1}} = \sum_{\abs{k} \geq M}\frac{\abs{\hat k}^{s'}}{\abs{k}^{s'}} \leq M^{-s'}\sum_{\abs{k} \geq M}\abs{\hat h(k)}^{s'}= M^{-s'}\norm{\hat h}_{\ell^{s'}}^{s'}.
        \label{eq: low freq contamination - 2}
    \end{equation}
    The conclusion follows from \eqref{eq: low freq contamination - 0}--\eqref{eq: low freq contamination - 2}. 
\end{proof}

\subsection{Two examples}
\label{sect:zero count}

\begin{example}\label{}
    Let $0  < \eps \ll 1$ and consider the periodic bump function
    \begin{equation*}
        F_\eps(x) = \eps \exp{\brac{\frac{\cos 2\pi x - 1}{\eps^2}}}
    \end{equation*}
    and set
    \begin{equation*}
        g_\eps(x) = F_\eps'(x) = -\frac{2\pi}{\eps}\sin(2\pi x)\exp\brac{\frac{\cos(2\pi x) - 1}{\eps^2}}.
    \end{equation*}
    For $k \in \mathbb N > 0$, define
    \begin{equation*}
        \phi_{k,\eps}(x) := g_\eps(kx).
    \end{equation*}
    A routine calculation shows that $\int_\T g_\eps = 0$ and $g_\eps$ has exactly two zeros at $0$ and $\frac{1}{2}$. Moreover, 
    \begin{equation}
        \text{$\norm{g_\eps}_p \approx \eps^{1/p}$ for $1 \leq p \leq \infty$ \quad and \quad $\norm{g_\eps}_{\dot W^{-1,s}}\approx \eps^{1+1/s}$.}
        \label{eq:g eps}
    \end{equation}
    
    Since $\phi_{k,\eps}$ goes around the circle $k$ times and all zeroes are nondegenerate, the number of zeroes of $\phi_{k,\eps}$ is exactly $2k$. Choose the scale $\eps$ and the frequency $k$ simultaneously
    \begin{equation*}
        \eps_k = \frac{1}{k^2}
    \end{equation*}
    so that
    \begin{equation*}
        \phi_k(x) := -2k^2\pi\sin(2k\pi x)\exp\left[k^4(\cos(2 k \pi  x) - 1)\right].
    \end{equation*}

    To begin with the Sturm-Hurwitz estimate, we claim that the Fourier expansion begins at frequency $k$. To see this, we begin with $F_\eps(x)$ and notice that
    \begin{equation}
        \exp(\eps^{-2}\cos(2\pi x)) = I_0(\eps^{-2}) + 2\sum_{m=1}^\infty I_m(\eps^{-2})\cos(2m\pi x)
        \label{eq:Bessel}
    \end{equation}
    where $I_m$ is the modified Bessel function of the first kind. Differentiating \eqref{eq:Bessel} gives a sine series whose first coefficient is nonzero. After the composition $x \mapsto kx$, the first nonzero Fourier frequency of $\phi_k$ is precisely $k$, so the Sturm-Hurwitz theorem gives 
    \begin{equation*}
        N_{\phi_k}^{(\text{SH})} \geq 2k
    \end{equation*}
    which is sharp.

    On the other hand, plugging \eqref{eq:g eps} into \eqref{eq:Steinerberger} gives the estimate
    \begin{equation*}
        N_{\phi_{k}}^{\eqref{eq:Steinerberger}} \gtrsim \eps^{1/2}k = 1.
    \end{equation*}
    Plugging \eqref{eq:beta critical intro} with $(r,s) = (\infty,2)$ gives
    \begin{equation*}
        N_{\phi_{k}}^{\eqref{eq:beta critical intro}} \gtrsim k. 
    \end{equation*}

\end{example}

    \begin{example}
        We modify the exponential part of the example above and define
        \begin{equation*}
            \psi_k(x) = -2k^2\pi\sin(2k\pi x)\exp\left[k^4(\cos(2 \pi  x) - 1)\right].
        \end{equation*}
        The exponential part is strictly positive, so the zeros are induced by $\sin(2k\pi x)$, and there are exactly $2k$ of them, but the majority of them correspond to small amplitudes. 

        We begin with the Sturm-Hurwitz estimate. Set $\theta:=2\pi x$ and $a := k^4$, so 
        \begin{equation*}
            \psi_k(x) = (-2k^2\pi )e^{-a}q_k(\theta)
            \text{\quad where \quad}
            q_k(\theta):= \sin(k\theta)e^{a\cos\theta}.
        \end{equation*}
        Using the modified Bessel functions of the first kind as in the previous example,
        \begin{equation*}
            \begin{split}
                q_k(\theta) &= I_0(a)\sin(k\theta) +2\sum_{m=1}^\infty I_m(a)\sin(k\theta)\cos(m\theta)
                \\
                &= I_0(a)\sin(k\theta) + \sum_{m=1}^\infty
                I_m(a)[
                \sin((k+m)\theta) + \sin((k-m)\theta)
                ].
            \end{split}
        \end{equation*}
        Note that there are exactly two terms that can contribute to frequency $1$ in the sum, and they are given by $m = k\pm 1$, so the Fourier coefficient corresponding to $\sin(2\pi x)$ in $q_k$ is given by
        \begin{equation*}
            I_{k-1}(a) - I_{k+1}(a) = I_{k-1}({k^4}) - I_{k+1}(k^4) = \frac{2k}{k^4}I_k(k^4) > 0,
        \end{equation*}
        where we use the Bessel recurrence relation in the second equality. Therefore, by the Sturm-Hurwitz theorem,
        \begin{equation*}
            N_{\phi_k}^{\text{(SH)}} \geq 2.
        \end{equation*}

        Now we move on to estimates using \eqref{eq:Steinerberger} and \eqref{eq:beta critical intro}. We only consider the cases where $k$ is sufficiently large. 

        Near $x = 0$, we have
        \begin{equation*}
            k^4(\cos(2\pi x) - 1) = -2\pi^2 k^4 x^2 + O(k^4x^4)
        \end{equation*}
        so the majority of the mass of the exponential function is concentrated in a region with width $\approx k^{-2}$. On the other hand,
        \begin{equation*}
            \sin(2k\pi x) = 2k \pi x + O(k^3\abs{x}^3).
        \end{equation*}
        As a consequence, 
        \begin{equation}
            \norm{\psi_k}_{p} \approx _p k^{1-2/p}
            \text{\quad for $1\leq p \leq  \infty$.}
            \label{eq:Psik-Lp}
        \end{equation}
        We claim that 
        \begin{equation}
            \norm{\psi_k}_{\dot H^{-1}} \approx k^{-2}.
            \label{eq:Psik-H1}
        \end{equation}
        Assuming the validity of \eqref{eq:Psik-H1}, we plug \eqref{eq:Psik-Lp} and \eqref{eq:Psik-H1} into \eqref{eq:Steinerberger} to obtain
        \begin{equation*}
            N_{\phi_k}^{\eqref{eq:Steinerberger}} \gtrsim k^{-1}
        \end{equation*}
        which is asymptotically vacuous. On the other hand, \eqref{eq:beta critical intro} gives
        \begin{equation*}
            N_{\phi_k}^{\eqref{eq:beta critical intro}} \gtrsim 1.
        \end{equation*}

        To show \eqref{eq:Psik-H1}, consider the rescaled variable
        \begin{equation*}
            x = \frac{y}{k^2}.
        \end{equation*}
        Then $\sin(2\pi kx) = 2k^{-1}\pi y + O(k^{-3}\abs{y}^3)$ and $k^4(\cos(2\pi x)-1) = -2\pi^2y^2 + O(k^{-4}y^4)$. Therefore, for a fixed $y$,
        \begin{equation}
            \lim_{k\to\infty}\frac{1}{k}\psi_k\brac{\frac{y}{k^2}} = -4\pi^2y e^{-2\pi^2 y^2}.
            \label{eq: convergence k - 1}
        \end{equation}
        Heuristically speaking, $\psi_k$ has height $\approx k$, total mass $\approx k^{-1}$, with majority of the mass concentrated in an interval of width $\approx k^{-2}$.

        For convenience, identify $\T \cong [-1/2,1/2)$ and let 
        \begin{equation}
            P_k(x) := \int_{-1/2}^x \psi_k(\xi)d\xi
            \quad \implies \quad
            kP_k\brac{\frac{y}{k^2}} = \int_{-k^2/2}^y \frac{1}{k}\psi_k\brac{\frac{z}{k^2}}dz.
            \label{eq: convergence k - 2}
        \end{equation}
        Note that $P_k$ is periodic since $\psi_k$ is odd. Using \eqref{eq: convergence k - 1}, \eqref{eq: convergence k - 2}, we have
        \begin{equation}
            \lim_{k\to\infty} kP_k\brac{\frac{y}{k^2}} = \int_{-\infty}^y 2\pi^2z e^{-2\pi^2z^2}dz = e^{-2\pi^2y^2}.
            \label{eq:limit Pk}
        \end{equation}
        Heuristically speaking, for large $k$, $P_k$ has height $\approx k^{-1}$ and the majority of its mass is concentrated on an interval of width $\approx k^{-2}$. This gives
        \begin{equation}
            \int_\T P_k \approx k^{3}.
            \label{eq:Pk average}
        \end{equation}
        Now we estimate the $L^2$ mass of $P_k$ for large $k$:
        \begin{equation}
            \norm{P_k}_2^2 = \frac{1}{k^2}\int_{-k^2/2}^{k^2/2}
            \abs{P_k\brac{\frac{y}{k^2}}}^2 dy \to \frac{1}{k^4}\int_\R e^{-4\pi^2y^2}dy \approx k^{-4}.
            \label{eq:Pk L2}
        \end{equation}
        The convergence is justified since $\abs{kP_k\brac{\frac{y}{k^2}}} \leq C_1e^{-C_2 y^2} \in L^1$ for $\abs{y}\leq k^2/2$.
        
        Finally, since we are working in $L^2$, \eqref{eq:Pk average} and \eqref{eq:Pk L2} implies
        \begin{equation*}
            \norm{\psi_k}_{\dot H^{-1}}^2 = \inf_{c\in \R}\norm{P_k - c}_2^2 = 
            \norm{P_k}_2^2 - \abs{\int_\T P_k}^2 \approx k^{-4}. 
        \end{equation*}
        Hence, \eqref{eq:Psik-H1} holds.

    \end{example}

\subsection{The potential $n$-dimensional problem}
\label{sect:higher dimensional}

Let $\T^n \cong [0,1)^n$ denote the normalized $n$-dimensional torus, and let $\phi:\T^n \to \R$ with $\int_{\T^n}\phi = 0$. We seek a bound of the form \eqref{eq:beta critical intro} with suitable modifications. In~\cite{steinerberger-interface} and a subsequent improvement~\cite{carroll2020enhanced}, the quantity $\#\set{\phi = 0}$ from \eqref{eq:Steinerberger} was replaced by $\mathcal H^{n-1}(\set{\phi = 0})$, and the authors showed that
\begin{equation*}
    \mathcal H^{n-1}(\set{\phi = 0}) \cdot \mathcal W_1(\phi_+,\phi_-) \gtrsim_n \brac{\frac{\norm{\phi}_{L^1(\T^n)}}{\norm{\phi}_{L^\infty(\T^n)}}}^{2-1/n}\norm{\phi}_{L^1(\T^n)}
\end{equation*}
with reasonable assumption on $\set{\phi=0}$.
Here and below, $\mathcal H^{n-1}$ 
is the $(n-1)$-dimensional Hausdorff measure and $\mathcal W_1(\phi_+,\phi_-)$ 
is the $1$-Wasserstein distance between the positive and negative parts of $\phi$, which coincides with $\norm{\phi}_{\dot W^{-1,1}}$. It is also shown in \cite{steinerberger2020metric}, using a purely planar argument, that the exponent $2-1/n$ can be replaced by $1$ when $n = 2$, which is sharp for $s=1$ in view of Theorem \ref{thm:critical L1Lr} and Remark \ref{rem:sharp critical cases}. The sharpness has also been pointed out in~\cite[Proposition 5]{carroll2020enhanced}.
% For $n \geq 3$, the exponent $4-1/n$ has been improved to $2-1/n$ in~\cite{carroll2020enhanced}. 

To formalize our question, we first replace $\mathcal H^{n-1}(\set{\phi = 0})$ by a more measure-theoretically stable notion, namely, the perimeter
\begin{equation*}
    \mathsf{Per}\set{\phi > 0} = \mathcal H^{n-1}(\partial^e\set{\phi > 0}),
\end{equation*}
where $\partial^e$ is the essential (or measure-theoretic) boundary
\begin{equation*}
    \partial^e S := \set{x \in \T^n : \limsup_{r \downarrow 0}
    \frac{\abs{S\cap B(x,r)}}{\abs{B(x,r)}} > 0 \text{ and }
    \limsup_{r \downarrow 0}
    \frac{\abs{S^c\cap B(x,r)}}{\abs{B(x,r)}} > 0
    }.
\end{equation*}

\begin{problem}
Let $1\leq s < \infty$ and $1 < r\leq \infty$.
Does the estimate
\begin{equation}
    \brac{\mathsf{Per}\set{\phi > 0} + \mathsf{Per}\set{\phi < 0}} \cdot \norm{\phi}_{\dot W^{-1,s}(\T^n)} \gtrsim_{n,r,s}
    \brac{\frac{\norm{\phi}_{L^1(\T^n)}}{\norm{\phi}_{L^r(\T^n)}}}^{r'/s}\norm{\phi}_{L^1(\T^n)}
    \label{eq:nd discussion}
\end{equation}
hold uniformly for all zero-mean $\phi \in C(\T^n)$?
\end{problem}

% We note that the suggested scaling in \eqref{eq:nd discussion} is again determined by testing a function of the form $\phi_\eps = \psi(\eps^{-1}(x-x_0))$ where $\psi \in C^\infty$ is compactly supported in $(0,1)^n$ with well behaved $\set{\psi > 0}$. We then have $\norm{\phi_\eps}_p \approx A\eps^{n/p}$, $\norm{\phi_\eps}_{\dot W^{-1,s}}\approx A\eps^{1+n/s}$, and $\mathsf{Per}\set{\phi > 0} \approx \eps^{n-1}$.  

% As mentioned earlier, the case $(n,s)=(2,1)$ has been handled by \cite{steinerberger2020metric} for nondegenerate $\phi$. However, the general case may require substantially new insight.

\bibliographystyle{plain}
\bibliography{Fourier}

\end{document}